\documentclass[11pt]{article}
\usepackage{amsmath,amsthm}
\usepackage{amssymb,hyperref,graphicx,color,array,fullpage}
\usepackage[notref,notcite,color,final]{showkeys}

\usepackage{enumerate}

\makeatletter
\def\author#1{\gdef\autrun{\def\and{\unskip, }#1}\gdef\@author{#1}}

\makeatother

\def\subjclass#1{{\renewcommand{\thefootnote}{}%
\footnote{\emph{Mathematics Subject Classification (2010):} #1}}}

\newtheorem{theorem}{Theorem}[section]
\newtheorem{lemma}[theorem]{Lemma}

\newtheorem{definition}[theorem]{Definition}

\newtheorem{notation}[theorem]{Notation}

\newtheorem{proposition}[theorem]{Proposition}

\def\p{\varphi}

\begin{document}

\title{\textsc{Michael-Simon inequalities for $k$-th mean curvatures}}
%\title{\textsc{Michael-Simon Inequalities for $k$-curvature operators}}
\author{\textsc{Yi Wang\thanks{Department of Mathematics, Stanford University, Stanford, CA 94305 Email: wangyi@math.stanford.edu;
the research is partially supported by NSF grant DMS-1205350.}}}
\date{}

\maketitle

%\address{Sun-Yung A. Chang, Princeton University, Department of Mathematics, Princeton, NJ 08540 Email: chang@math.princeton.edu}

%\footnote{1991 Mathematics Subject Classification. Primary 53A30; Secondary 42B35}
%\email{chang@math.princeton.edu}
%\email{wangyi@math.princeton.edu}
%\author{\textsc{Sun-Yung Alice Chang* and Yi Wang}}
%\subjclass{Primary 53A30; Secondary 42B35}

\begin{abstract}
This paper continues the study of Alexandrov-Fenchel inequalities for quermassintegrals for $k$-convex domains. It focuses on the application to the Michael-Simon type inequalities for $k$-curvature operators. The proof uses
optimal transport maps as a tool to relate curvature quantities defined on the boundary of a domain.
%We have also show the condition $k+1$-convexity is optimal. More precisely, %for $k$-convex hypersurfaces the inequality is not true.
\end{abstract}

\section{Introduction}

The classical Michael-Simon inequality is the Sobolev inequality on immersed submanifolds.

\begin{theorem}\cite{MS} Let $i: M^n\to \mathbb{R}^{N}$ be an isometric immersion ($N>n$). Let $U$ be an open subset of $M$.
For a function $\p\in C^{\infty}_c(U)$, there exists a constant $C$, such that
\begin{equation}\label{ms0}
\displaystyle \left(\int_{M}|\p|^{\frac{n}{n-1}}d\mu_M\right)^{\frac{n-1}{n}}\leq C\int_{M} |\vec H|\cdot |\p|+|\nabla \p|dv_{M}.
\end{equation}
\end{theorem}
In the special case when we take $\p\equiv 1$, Michael-Simon inequality gives an inequality between the area of the boundary
and the integral of its mean curvature.
In this note, we derive a natural generalization of (\ref{ms0}) in which we establish inequalities between fully nonlinear curvature quantities $\sigma_{k-1}(L)$ and $\sigma_{k}(L)$ if the hypersurfaces $M$ is $(k+1)$-convex, where $\sigma_k(L)$ denotes the $k$-th elementary symmetric function of the second fundamental form $L$.

\begin{theorem}\label{MS}Let $i: M^n\to \mathbb{R}^{n+1}$ be an isometric immersion. Let $U$ be an open subset of
 $M$ and $\p$ be a $C^{\infty}_c(U)$ function. For $k=2,..., n-1$, if $M$ is $(k+1)$-convex,
 then for any $1\leq l\leq k-1$, there exists a constant $C$ depending only on $n$ $k$ and $l$, such that
$$\displaystyle\left( \int_{M} \sigma_{l}(L)|\p|^{\frac{n-l}{n-k}}d\mu_M\right)^{\frac{n-k}{n-l}}\leq C
 \int_{M} (\sigma_{k}(L)|\p|+\sigma_{k-1}(L)|\nabla \p|+,...,+|\nabla^k \p| )d\mu_M.\\$$
If $k=n$, then the inequality holds when $M$ is $n$-convex. If $k=1$, then the inequality holds when $M$ is $1$-convex.
($k=1$ $l=0$ case is the Michael-Simon inequality.)
\end{theorem}

Theorem \ref{MS} generalizes previous works \cite{Chang-Wang2010} \cite{CW}
on the Alexandrov-Fenchel inequality for quermassintegrals of $k+1$-convex domains.
% in the sense that when we set $\p \equiv 1$ in the statement of the theorem, we obtain the main result (Theorem 1.3) in \cite{CW}.
The proof of Theorem \ref{MS} follows from the same outline as
that of Theorem 1.3 in \cite{CW}; with the added complication
of the present of weights $\p$ and its higher order derivatives.
%which is inspired by the work of P. Castillon \cite{Castillon}.
%It uses the theory of optimal transport, Garding's theory for hyperbolic polynomials, together with delicate induction argument on the polarized curvature terms $\Sigma_{k}(\overbrace{D^2v,...,D^2v }^m, L,..., L)$.
The main technical part lies in the proof of Proposition \ref{a}. We reduce the proof of Proposition \ref{a}
into four types of estimates, which are defined to be the $I$-type, the $J$-type, the $K$-type and the $N$-type estimate (in Section 5). Among them, $K$-type estimate is quite different from the one in \cite{CW}, and $N$-type estimate is new.
%The first two types of estimates are very similar to those in \cite{CW} before, while the last two types, the $K$-type and the $N$-type, are new.
In the proof we will briefly go through the $I$-type estimate and the $J$-type estimate which are similar to those in \cite{CW}; we then focus on the $K$-type estimate and the $N$-type estimate, especially the interplay between them.

The organization of this paper is as follows. In Section 2, we will recall some preliminary facts on elementary symmetric functions and curvature properties of embedded hypersurfaces.
%Since many of them have been proved in \cite{CW}, we will not repeat the proofs.
In Section 3, we will demonstrate the method of optimal transport and reduce the proof of Theorem \ref{MS} to the technical proposition (Proposition \ref{a}).
%As in \cite{CW}, the proof of the main theorem follows the outline similar to that in the paper by P. Castillon, but we deal with the fully nonlinear quantities of the curvature.
In Section 4, we will present the proof of Proposition \ref{a}
for the special cases $k=2$. In Section 5, we will prove Proposition \ref{a} for general $k$ by a delicate induction argument.

We remark that it is an open question to prove Michael-Simon inequality (\ref{ms0}) with sharp constant.
The Michael-Simon inequality for higher order curvatures we derive in this paper does not yield any
sharp constants either.

{\bf Acknowledgments:}
The author would like to thank Professor Sun-Yung Alice Chang for many discussions at the early developing stage of the work.

\section{Preliminaries}
%Notation: Let us denote $V_\Omega: \mathbb{R}^n\mapsto \mathbb{R}^n$.  $\bar{V}|_{M}:=V_{\Omega}\circ p $. $v:=\bar{V}|_{M}$.  $D^2 v$ is the Hessian of $V$ on $M$. $ [T_{k}]_{ij}(\overbrace{v,...v}^m,\overbrace{L,...,L}^{k-m})$.

\subsection{\textbf{$\Gamma_k^+$}cone}
In this subsection, we will describe some properties of the $k$-th elementary symmetric function $\sigma_k$ and its associated convex cone.
\subsubsection{Definitions and basic properties}\begin{definition} The $k$-th elementary symmetric function for
 $\lambda=(\lambda_1,...,\lambda_n)\in \mathbb{R}^n$ is
$$\sigma_{k}(\lambda):=\sum_{i_1<...<i_k}\lambda_{i_1}\cdots \lambda_{i_k}.$$
\end{definition}
The elementary symmetric functions are special cases of hyperbolic polynomials introduced by Garding
 \cite{Garding}, which enjoy the following properties in their associated positive cones.
\begin{definition}\label{cone}
$$\Gamma_{k}^{+}:=\{\lambda\in \mathbb{R}^n|\mbox{the connected component of } \sigma_{k}(\lambda)>0
\mbox{ which contains the identity}=(1,...,1)\}$$ is called the positive $k$-cone.\\
Equivalently, $$\Gamma_k^+=\{\lambda\in \mathbb{R}^n|\ \sigma_{1}(\lambda)>0,...,\sigma_{k}(\lambda)>0\}.$$
\end{definition}
\noindent In particular, $\Gamma_{n}^+$ is the positive cone $$\{\lambda\in \mathbb{R}^n|\ \lambda_1> 0,..., \lambda_n> 0  \},$$
and $\Gamma_{1}^+$ is the half space $\{\lambda\in \mathbb{R}^n| \lambda_1+\cdots +\lambda_n> 0 \}$.
It is also obvious from Definition \ref{cone} that $\Gamma_k^+$ is an open convex cone and that
$$\Gamma_{n}^+ \subset \Gamma_{n-1}^+\cdots \subset\Gamma_{1}^+.$$

\noindent By Garding's theory of hyperbolic polynomials \cite{Garding}, one concludes that
$\sigma_{k}^{\frac{1}{k}}(\cdot)$ and $(\frac{\sigma_{k}(\cdot)}{\sigma_{l}(\cdot)})^{\frac{1}{k-l}}$ ($k>l$) are concave functions in $\Gamma_{k}^{+}$.
%Thus
%\begin{equation}\label{2.4}
%\frac{\sigma_k^{\frac{1}{k}}(\lambda)+\sigma_k^{\frac{1}{k}}(\mu)}{2}\leq %\sigma_k^{\frac{1}{k}}\left(\frac{\lambda+\mu}{2}\right),
%\end{equation}
%for $\lambda, \mu\in \Gamma_{k}^{+}$.
%By the homogeneity of $\sigma_k^{\frac{1}{k}}$, one gets from (\ref{2.4}) that for $\lambda, \mu\in \Gamma_{k}^{+}$
%\begin{equation}
%\sigma_k^{\frac{1}{k}}(\lambda)<\sigma_k^{\frac{1}{k}}(\lambda+\mu).
%\end{equation}
%Also, $(\frac{\sigma_{k}(\cdot)}{\sigma_{l}(\cdot)})^{\frac{1}{k-l}}$ ($k>l$) is concave in $\Gamma_{k}^{+}$.
%Therefore
%\begin{equation}
%(\frac{\sigma_{k}(\lambda)}{\sigma_{l}(\lambda)})^{\frac{1}{k-l}}< (\frac{\sigma_{k}(\lambda+\mu)}
%{\sigma_{l}(\lambda+\mu)})^{\frac{1}{k-l}},
%\end{equation}
%for $\lambda, \mu\in \Gamma_{k}^{+}$.\\

\begin{definition}A symmetric matrix $A$ is in $\tilde{\Gamma}_{k}^+$ cone, if its
eigenvalues $$\lambda(A)=(\lambda_1(A),...,\lambda_n(A))\in  \Gamma_{k}^+.$$
\end{definition}
%Suppose $f$ is a function on $\Gamma_{k}^+$. $F=f(\lambda(A))$ is the extension of $f$ on $\tilde{\Gamma}_{k}^+$.
%Due to a result in \cite{Caff4}, $f$ is concave in $\Gamma_{k}^+$ implies $F$ is concave in $\tilde{\Gamma}_{k}^+$.
When there is no confusion, we will denote $\tilde{\Gamma}_{k}^+$ by $\Gamma_{k}^+$ and
$\sigma_{k}(\lambda(A))$ by $\sigma_{k}(A)$ for simplicity.
%Notice that $\sigma_n(A)=\det (A)$. An equivalent definition of $\det (A)$ is
%\begin{equation}\label{det}\det A=\frac{1}{n!}\delta^{i_1,...,i_n}_{j_1,...,j_n}A_{i_1j_1} \cdots A_{i_nj_n},\end{equation}
%where $\delta^{i_1,...,i_n}_{j_1,...,j_n} $ is the generalized Kronecker delta; it is zero if %$\{i_1,...,i_k\}\neq\{j_1,...,j_k\}$,
%equals to 1 (or -1) if $(i_1,...,i_k)$ and $(j_1,...,j_k)$ differ by an even (or odd) permutation.
An equivalent definition of $\sigma_k(A)$ is
$$\sigma_{k}(A):=\frac{1}{k!}\delta^{i_1,...,i_k}_{j_1,...,j_k}A_{i_1j_1} \cdots A_{i_kj_k}.$$
The Newton transformation tensor is defined as
\begin{equation}\label{Newton}[T_{k}]_{ij}(A_1,...,A_k):=\frac{1}{k!}
\delta^{i,i_1,...,i_k}_{j,j_1,...,j_k}{(A_1)}_{i_1j_1} \cdots {(A_k)}_{i_kj_k}.\end{equation}
\begin{definition}\label{2.3}
With the notion of $[T_{k}]_{ij}$, one may define the polarization of $\sigma_k$ by
\begin{equation}\Sigma_{k}(A_1,...,A_k):={A_1}_{ij}\cdot[T_{k-1}]_{ij}(A_2,...,A_k)=
\frac{1}{(k-1)!}\delta^{i_1,...,i_k}_{j_1,...,j_k}{(A_1)}_{i_1j_1} \cdots {(A_k)}_{i_kj_k}.\end{equation}
\end{definition}
We remark here that $\Sigma_{k}(A,...,A)$ and $\sigma_k(A)$ only differs by a multiplicative constant.
$$\sigma_{k}(A)=\frac{1}{k}\Sigma_{k}(A,...,A).$$
Therefore it is called the polarization of $\sigma_k$.
%Also, from the right hand side of the definition \ref{2.3}, we see that $\Sigma_{k}$ is symmetric and linear in each component.
\begin{notation}
When some components are the same, we adopt the notational convention that
$$\Sigma_{k}(\overbrace{B,...,B}^{l},C,...,C):=\Sigma_{k}(\overbrace{B,...,B}^{l},\overbrace{C,...,C}^{k-l}),$$
and
$$[T_{k}]_{ij}(\overbrace{B,...,B}^{l},C,...,C):=[T_{k}]_{ij}(\overbrace{B,...,B}^{l},\overbrace{C,...,C}^{k-l}).$$
Also for simplicity, we denote
$$[T_{k}]_{ij}(A):=[T_{k}]_{ij}(\overbrace{A,...,A}^{k}).$$
\end{notation}
Some relations between the Newton transformation tensor $T_{k}$ and $\sigma_k$ are listed below.
For any symmetric matrix $A$, if we denote the trace of a matrix by $Tr(\cdot)$, then
\begin{equation}\label{Trace1}\sigma_{k}(A)=\frac{1}{n-k}Tr([T_{k}]_{ij})(A),\end{equation}
and
\begin{equation}\label{Trace2}\sigma_{k+1}(A)=\frac{1}{k+1}Tr([T_{k}]_{im}(A)\cdot A_{mj} ).\end{equation}
On the other hand, one can write $[T_k]_{ij}$ in terms of $\sigma_k$ by the formula
$$[T_{k-1}]_{ij}(A)=\frac{\partial \sigma_{k}(A)}{\partial A_{ij}},$$
and
\begin{equation}\label{T}[T_{k}]_{ij}(A)=\sigma_{k}(A)\delta_{ij}-[T_{k-1}]_{im}(A)A_{mj}.
\end{equation}
\noindent This last formula implies the following fact which we will repeatedly use later in our proof.
%\begin{lemma}\label{2.1T1}Suppose $B$ and $C$ are two symmetric matrices, then
%\begin{equation}\label{T1}
%\begin{split}
%&[T_{k-1}]_{im}(B,C,...,C)C_{mj}\\
%=& \displaystyle \frac{1}{k-1}\Sigma_{k}(B,C,...,C)\delta_{ij}-\frac{k}{k-1}[T_{k}]_{ij}
%(B,C,...,C)-\frac{1}{k-1}[T_{k-1}]_{im}(C,...,C)B_{mj}.\\
%\end{split}
%\end{equation}
%\end{lemma}
%\noindent By a similar argument, one has
\begin{lemma}\label{2.2Tk}Suppose $B$ and $C$ are two symmetric matrices, then
\begin{equation}\label{T_k}
\begin{array}{lcl}&&[T_{k-1}]_{im}(\overbrace{B,...,B}^{l},C,...,C)C_{mj}\\
&=&\displaystyle \frac{C_{k}^{l}}{kC_{k-1}^{l}} \cdot\Sigma_{k}(\overbrace{B,...,B}
^{l},C,...,C)\delta_{ij} -\frac{C_{k}^{l}}{C_{k-1}^{l}}\cdot[T_{k}]_{ij}(\overbrace{B,...,B}^{l},C,...,C)\\
&&\displaystyle -\frac{C_{k-1}^{l-1}}{C_{k-1}^{l}}\cdot[T_{k-1}]_{im}(\overbrace{B,...,B}^{l-1},C,...,C)B_{mj}.\\
\end{array}
\end{equation}
\end{lemma}
\noindent We omit the proof here since it is quite straightforward by formula (\ref{T}) and the multi-linearity of $[T_{k}]$ and $\sigma_k$. One can also refer to Lemma 2.7
%\ref{2.2Tk}
in \cite{CW} for a complete proof.\\

We finish this section by listing some basic inequalities based on the Garding's theory of hyperbolic polynomials \cite{Garding},
which we will use in the present paper.\\

(i) if $\lambda\in \Gamma_k^+$, then
$$\frac{\partial \sigma_{k}(\lambda)}{\partial \lambda_{i}}>0, \mbox{ for } i=1,...,n;$$

(ii) if $A_1,...,A_{k}\in \Gamma_{k+1}^+$, then $([T_{k}]_{ij})$ is a positive matrix, i.e.\\
 $$[T_{k}]_{ij}(A_1,...,A_k)> 0;$$

(iii) if $A_1,...,A_{k}\in \Gamma_k^+$, then
 $$ \Sigma_{k}(A_1,...,A_k)> 0;$$

(iv) if $A-B\in \Gamma_k^+$ and $A_2,...,A_k\in \Gamma_k^+ $, then
$$ \Sigma_{k}(B,A_1...,A_k)<  \Sigma_{k}(A,A_2,...,A_k).$$

%\noindent Lastly, for nonnegative symmetric matrix $A$, we have the well-known Newton-MacLaurin inequality: (see e.g. \cite{HL})

%\begin{equation}
%\quad \quad \frac{\sigma_{k+1}(A)\sigma_{k-1}(A)}{\sigma_{k+1}(Id)\sigma_{k-1}(Id)}\leq \frac{\sigma^2_{k}(A)}{\sigma^2_{k}(Id)},
%\end{equation}
%where $Id$ is the identity matrix.

Finally, we recall two technical lemmas regarding the derivative of the Newton transformation tensor $T_{k}$.

\begin{lemma}\label{divT}
Let $L$ denote the second fundamental form of the hypersurface $M^n\hookrightarrow \mathbb{R}^{n+1}$. Let $[T_{k}]_{ij}(L)$ be the Newton transform tensor of $L$.
Then the divergence of $[T_{k}]_{ij}(L)$ is equal to 0, i.e.
\begin{equation}
([T_{k}]_{ij}(L))_i=0 \quad \mbox{for each $j$}.
\end{equation}
\end{lemma}
\noindent The proof uses the Codazzi equation
\begin{equation}
\label{Codazzi}
L_{ij,k}=L_{ik,j}
\end{equation}
 and properties of $[T_k]$. We refer interested reader to see Lemma 5.1
%\ref{4.11}
in \cite{CW}.

\begin{lemma}\label{divTD2v}
Suppose $v$ is a smooth function defined on the hypersurface $M^n\hookrightarrow \mathbb{R}^{n+1}$. Denote the Hessian of $v$ on $M$ by $D^2 v$, the second fundamental form of $M$ by $L$. Consider the polarized Newton transformation tensor
$[T_{k}]_{ij}(\overbrace{D^2v, ..., D^2v}^{m},L,...,L)$ introduced in Definition \ref{Newton}. The divergence of it satisfies
\begin{equation}
([T_{k}]_{ij}(\overbrace{D^2v, ..., D^2v}^{l},L,...,L))_i= [T_{k}]_{ij}(\overbrace{D^2v, ..., D^2v}^{l-1},L,...,L)  L_{mi} v_m
\quad \mbox{for each $j$}.
\end{equation}
\end{lemma}
The proof of this lemma uses the above Codazzi equation and the Gauss equation
\begin{equation}\label{Gauss}
0=\bar{R}_{ijkl}=R_{ijkl}-L_{ik}L_{jl}+L_{il}L_{jk}, \quad \mbox{(Gauss equation)}
\end{equation}
were the curvature tensor of
$M$ and the curvature tensor of the ambient space $\mathbb{R}^{n+1}$ are denoted by $R_{ijkl}$  and by $\bar{R}_{ijkl}$ respectively.
The detailed proof has appeared in (120)-(121)
 %(\ref{5.25})-(\ref{5.39})
of \cite{CW}.

\subsection{Restriction of a convex function to a submanifold}
Consider an isometric immersion $i:M^{n}\hookrightarrow \mathbb{R}^{n+1}$. Let $\nabla$ and $D^2$ (resp. $\bar{\nabla}$ and $\bar{D}^2$) be the gradient and the Hessian on $M$ (resp. on $\mathbb{R}^{n+1}$). We also denote the second fundamental form on $ M$ by $L_{ij}$ and
the inner unit normal by $\vec n $. Suppose $\bar{V}:
\mathbb{R}^{n+1}\rightarrow \mathbb{R}$ is a smooth function and $v=\bar{V}|_{M}$ is its restriction
to $M$. Then the Hessian of $v$ with respect to the metric on $M$ relates to the Hessian of
$\bar{V}$ on the ambient space $\mathbb{R}^{n+1}$ by
\begin{equation}
\begin{split}
D^2_{ij}v=&\bar{D}^2_{ij}\bar{V}+\langle (\bar{\nabla} \bar{V}, \vec n)\rangle L_{ij}\\
=&\bar{D}^2_{ij}\bar{V}+b(x) \cdot L_{ij},\\
\end{split}
\end{equation}
where $b(x):=\langle (\bar{\nabla} \bar{V}), \vec n \rangle (x)$. We remark in general $b(x)$ changes sign on $M$ and $|b(x)|\leq |\bar{\nabla} \bar{V} |$.

\section{Proof of the main theorem}
\noindent\textbf{Theorem} \ref{MS} (Main Theorem): \textit{Let $i: M^n\hookrightarrow \mathbb{R}^{n+1}$ be an isometric immersion. Let $U$ be an open subset of
 $M$ and $\p$ be a $C^{\infty}_c(U)$ function. For $k=2,..., n-1$, if $M$ is $(k+1)$-convex,
 then for any $1\leq l\leq k-1$, there exists a constant $C$ depending only on $n$ $k$ and $l$, such that
$$\displaystyle\left( \int_{M} \sigma_{l}(L)|\p|^{\frac{n-l}{n-k}}d\mu_M\right)^{\frac{n-k}{n-l}}\leq C
 \int_{M} (\sigma_{k}(L)|\p|+\sigma_{k-1}(L)|\nabla \p|+,...,+|\nabla^k \p| )d\mu_M.\\$$
If $k=n$, then the inequality holds when $M$ is $n$-convex. If $k=1$, then the inequality holds when $M$ is $1$-convex.
($k=1$ $l=0$ case is a corollary of the Michael-Simon inequality.)
}

The main technical part of this paper is the following proposition (Proposition \ref{a}).

%DELETE this entire paragraph, and move the argument till later---
% In order to prove the main theorem, we will apply the following proposition (Proposition \ref{a}), whose proof in Section 6 is the main technical part of this paper. Because of the intricacy of the method to prove this proposition, we will demonstrate the idea of the proof by discussing the two simplest cases k=2 and k=3 in Section 4 and 5 respectively.

\begin{proposition}\label{a} Let $E\subset \mathbb{R}^{n+1}$ be an $n$-dimensional
linear subspace, and $p$ be the orthogonal projection from $\mathbb{R}^{n+1}$ to $E$.
Suppose $V: E\to \mathbb{R}$ is a $C^3$ convex function that satisfies $|\nabla V|\leq 1$.
Define its extension to $\mathbb{R}^{n+1}$ by $\bar{V}:=V\circ p $, and define the restriction of
$\bar{V}$ to the immersed hypersurface $M$ by $v$.
Suppose also that $M$ is $(k+1)$-convex if $2\leq k\leq n-1$, i.e.
the second fundamental form $L_{ij}\in \Gamma_{k+1}^{+}$. And suppose
that $M$ is $n$-convex if $k=n$. Then
for each $k$ and each constant $a>1$ and any $ C^{\infty}_c(U)$ function $\p$, there exists a constant $C$, which depends only on $k$, $n$ and $a$, such that
\begin{equation}\label{a1}
 \displaystyle \int_{M} \sigma_k(D^2v +a L)\p d\mu_M\leq C \int_{M} \sigma_k(L)|\p|+\sigma_{k-1}(L)|\nabla \p|+...+ |\nabla^k \p|d\mu_M.\\
\end{equation}
Note that $C$ does not depend on $v$.
\end{proposition}

Our proof of Proposition \ref{a} uses a multi-layer induction process and is quite complicated.
We will first illustrate the idea of the proof of the proposition for the (easy) case $k=2$ in Section 4,
and we will finish the proof for all integers $k$ in Section 5.
In the rest of this section, we will prove the main theorem assuming Proposition \ref{a}.
The proof follows the outline similar to that of the main theorem in \cite{CW} which is inspired by the work of P. Castillon \cite{Castillon}. Since such an argument is standard and has appeared with minor difference in \cite{CW} already, we will only describe the difference of its proof from the one in \cite{CW} without repeating the whole paragraph.

\begin{proof}[Brief outline of the Proof of Theorem \ref{MS}]
The differences of the proof is to first take a different function $f$ on $M$. Namely, instead of taking \begin{equation}
\displaystyle f:= \frac{\sigma_{k-1}(L)J_{E}^{\frac{1}{n-k}}}{\int_{M}\sigma_{k-1}(L)J_{E}^{\frac{1}{n-k}}d\mu_M}
\end{equation}
we define
\begin{equation}
\displaystyle f:= \frac{\sigma_{l}(L)|\p|^{\frac{n-l}{n-k}}J_{E}^{\frac{k-l}{n-k}}}{\int_{M}\sigma_{l}(L)|\p|^{\frac{n-l}{n-k}}J_{E}^{\frac{k-l}{n-k}}d\mu_M}.
\end{equation}
$f(x)dx$ is again a probability measure on $M$. Thus we follow the same argument to derive inequality (37) in \cite{CW}:
\begin{equation}\begin{split}\label{LHSRHS}
&\left(\omega_{n}f(x)J_{E}(x)\right)^{\frac{k-l}{n-l}}
\cdot \frac{\sigma_{l}(\bar{D}^2\bar{V}+(a-1)L)}
{\sigma_{l}(\bar{D}^2\bar{V})^{\frac{k-l}{n-l}}}\\
\leq &\displaystyle \left(det(\bar{D}^2\bar{V}(x))\right)^{\frac{k-l}{n-l}}\cdot
\frac{\sigma_{l}(\bar{D}^2\bar{V}+(a-1)L)}
{\sigma_{l}(\bar{D}^2\bar{V})^{\frac{k-l}{n-l}}} .
\end{split}
\end{equation}
Denote the left hand side (resp. right hand side) of this inequality by $LHS$ (resp. $RHS$).
By exactly the same argument using the method of optimal transport as in \cite{CW},
\begin{equation}\begin{split}
RHS\leq & C_{n,k}^{\frac{1}{n-k+1}} \sigma_{k}(D^2 v+aL);\\
\end{split}
\end{equation}
while on the other hand, by taking the newly defined function $f$, we obtain
\begin{equation}
\begin{split}
LHS \geq & \displaystyle \frac{(a-1)^{l\cdot(1-\frac{k-l}{n-l})}\omega_{n}^{\frac{k-l}{n-l}}\sigma_{l}(L)|\p|^{\frac{k-l}{n-k}}J_{E}^{\frac{k-l}{n-k}}}
{(\displaystyle \int_{M}\sigma_{l}(L)|\p|^{\frac{n-l}{n-k}}J_{E}^{\frac{k-l}{n-k}}d\mu_M)^{\frac{k-l}{n-l}}}.
\end{split}
\end{equation}
Now we multiply $|\p|$ on both $LHS$ and $RHS$, and integrate both of them over $M$.
This gives rise to
\begin{equation}\begin{split}
&\displaystyle \frac{\displaystyle(a-1)^{l\cdot(1-\frac{k-l}{n-l})}\omega_n^{\frac{k-l}{n-l}}\int_{M}
\sigma_{l}(L)|\p|^{\frac{n-l}{n-k}}J_{E}^{\frac{k-l}{n-k}}d\mu_M}{(\displaystyle \int_{M}
\sigma_{l}(L)|\p|^{\frac{n-l}{n-k}}J_{E}^{\frac{k-l}{n-k}}d\mu_M)^{\frac{k-l}{n-l}}}\\
\leq & C_{n,k}^{\frac{1}{n-k+1}} \int_{M} \sigma_{k}(D^2 v+aL)|\p|d\mu_M.\\
\end{split}
\end{equation}
This inequality plays the same role as inequality (47) in \cite{CW}. The argument after this inequality follows exactly in the same way as that in \cite{CW}. This finishes the brief description of the differences of the proof from the one in \cite{CW}.
\end{proof}

We remark here that regularity issue for optimal transport of non-convex domains will appear as it does in our previous paper \cite{CW}. Again, one can handle the problem using the approximation argument together with L. Caffarelli's regularity result ({\cite{Caff1}, \cite{Caff2}, \cite{Caff3}) for strictly convex domains. Such a method has also been demonstrated in \cite{CW} already, so we will not repeat it here.

\section{$k=2$ case of Proposition \ref{a} }
In this section, we are going to prove
\begin{equation}
 \displaystyle \int_{M} \sigma_2(D^2v+ a L) \p d\mu_M\leq C \int_{M} (\sigma_2(L) |\p|+ \sigma_1(L)|\nabla \p|+ |\nabla^2 \p|) d\mu_M.\\
\end{equation}
\begin{proof}
First of all, we can write
\begin{equation}\begin{split}\label{2.9}
 \displaystyle \int_{M} \sigma_2(D^2v+ a L) \p d\mu_M= &\displaystyle \int_{M} \frac{1}{2} \Sigma_2(D^2v+ a L,D^2v+ a L) \p d\mu_M\\
 =&\displaystyle  \int_{M} \frac{1}{2}\Sigma_2(D^2 v, D^2 v)\p+
 a\Sigma_2(D^2 v, L)\p+ \frac{a^2}{2} \Sigma_2 (L,L)\p  d\mu_M\\
 :=&\frac{1}{2}I + a \cdot II+ \frac{a^2}{2} III.\\
 \end{split}
\end{equation}

To bound the term I, by Definition 2.4 and the integration by parts formula
\begin{equation}\begin{split}\label{2.1}
I:=&\displaystyle\int_{M} \Sigma_2(D^2 v, D^2 v) \p d\mu_M \\
=&\displaystyle\int_{M}  v_{ij} [T_{1}]_{ij}(D^2 v) \p d\mu_M \\
=& \displaystyle\int_{M}  -v_{j} ([T_{1}]_{ij}(D^2 v))_i \p - v_{j} [T_{1}]_{ij}(D^2 v) \p_i d\mu_M \\
\end{split}
\end{equation}
For the first term, we apply the Riemannian curvature equation, $$([T_{1}]_{ij}(D^2 v))_i= v_{ii, j}- v_{ij,i}= R_{miij}v_{m}= (L_{mi}L_{ij}-L_{mj}L_{ii})v_{m}= -[T_{1}]_{ij}(L) L_{mi}v_{m}.$$
Thus
\begin{equation}\begin{split}\label{2.6}
&\displaystyle\int_{M}  -v_{j} ([T_{1}]_{ij}(D^2 v))_i \p d\mu_M \\
=& \displaystyle\int_{M}[T_{1}]_{ij}(L) L_{mi} v_{j} v_{m} \p d\mu_M \\
\end{split}
\end{equation}
By the assumption $L_{ij}\in \Gamma_3^+$, $[T_{1}]_{ij}(L) L_{mi}\leq \sigma_2(L)g_{ij}$.
In fact, one can diagonalize $L_{ij}\sim diag(\lambda_1, \cdots , \lambda_n)$; thus $[T_{1}]_{ij}(L) L_{mi}$ is also diagonalized,
$$[T_{1}]_{ij}(L) L_{mi}\sim diag(\lambda_1(\sigma_1(L)-\lambda_1), \cdots , \lambda_n(\sigma_1(L)-\lambda_n)).$$
We remark here that $[T_{1}]_{ij}(L) L_{mi}=L_{mj}L_{ii}-L_{mi}L_{ij}$ is a symmetric matrix.\\
Note that $$\lambda_i(\sigma_1(L)-\lambda_i)+ \frac{\partial \sigma_3(L)}{\partial\lambda_i} =\sigma_2(L), \quad
\mbox{for each $i$}.$$
 Also $L_{ij}\in \Gamma_3^+$ implies $\frac{\partial \sigma_3(L)}{\partial\lambda_i}\geq 0$. Thus $\lambda_i(\sigma_1(L)-\lambda_i)\leq \sigma_2(L)$ for each $i$. Therefore $[T_{1}]_{ij}(L) L_{mi}\leq \sigma_2(L)g_{ij}$.
Applying this to (\ref{2.6}), we get
\begin{equation}\begin{split}\label{2.7}
\displaystyle\int_{M}  -v_{j} ([T_{1}]_{ij}(D^2 v))_i \p d\mu_M
\leq  \displaystyle\int_{M} \sigma_2(L) |\nabla v|^2 \cdot |\p| d\mu_M\leq \displaystyle\int_{M} \sigma_2(L)  |\p| d\mu_M.\\
\end{split}
\end{equation}
with the last inequality following from $|\nabla v|\leq 1$.\\
For the second term $\int_{M}-v_{j} [T_{1}]_{ij}(D^2 v) \p_i d\mu_M$ in (\ref{2.1}), we use the relation
$D^2 v=\bar{D}^2\bar{V}+b(x)L$.
\begin{equation}\begin{split}
&\displaystyle\int_{M}-v_{j} [T_{1}]_{ij}(D^2 v) \p_i d\mu_M \\
=& \displaystyle\int_{M}-v_{j} [T_{1}]_{ij}(\bar{D}^2\bar{V}+b(x)L) \p_i d\mu_M \\
=&\displaystyle\int_{M}-v_{j} [T_{1}]_{ij}(\bar{D}^2\bar{V}) \p_i d\mu_M
+\displaystyle\int_{M}-v_{j} [T_{1}]_{ij}(b(x)L) \p_i d\mu_M.\\
\end{split}
\end{equation}
Since $[T_1]_{ij}(\bar{D}^2\bar{V})\geq 0$, $[T_{1}]_{ij}(L) \geq 0$, $|b(x)|\leq 1$, we have
$$-[T_1]_{ij}(\bar{D}^2\bar{V}) v_j \p_i\leq Tr([T_{1}]_{ij}(\bar{D}^2\bar{V})) |\nabla \p| \cdot |\nabla v|, $$
and $$-b(x)[T_1]_{ij}(L)v_j \p_i\leq Tr([T_{1}]_{ij} (L) )|\nabla \p|\cdot |\nabla v|,$$
where $Tr([T_{1}]_{ij}(\bar{D}^2\bar{V}))$ denotes the trace of $[T_1]_{ij}(\bar{D}^2\bar{V})$ and
$Tr([T_{1}]_{ij} (L) )$ denotes the trace of $[T_{1}]_{ij}(L)$.
Thus \begin{equation}\begin{split}\label{2.2}
&\displaystyle\int_{M} -v_{j} [T_1]_{ij}(\bar{D}^2\bar{V} +b(x)L) \p_i d\mu_M\\
\leq & \displaystyle\int_{M} Tr([T_{1}]_{ij}(\bar{D}^2\bar{V})) |\nabla \p| \cdot |\nabla v|+
Tr([T_{1}]_{ij} (L) )|\nabla \p|\cdot |\nabla v|d\mu_M\\
=&\displaystyle\int_{M}  (n-1)\sigma_{1}(\bar{D}^2\bar{V})|\nabla \p| \cdot |\nabla v|+(n-1)\sigma_{1}(L) |\nabla \p| \cdot |\nabla v|d\mu_M.\\
\end{split}
\end{equation}
Since $|\nabla v|\leq 1$,
\begin{equation}\label{2.3}\displaystyle\int_{M}  \sigma_{1}(L) |\nabla \p| \cdot |\nabla v|d\mu_M\leq\displaystyle\int_{M}  \sigma_{1}(L) |\nabla \p|d\mu_M.\end{equation}
On the other hand,
\begin{equation}\begin{split}\label{2.4}
&\displaystyle\int_{M}  \sigma_{1}(\bar{D}^2\bar{V}) |\nabla \p| \cdot |\nabla v|d\mu_M\\
\leq &\displaystyle\int_{M}  \sigma_{1}(\bar{D}^2\bar{V}) |\nabla \p| d\mu_M\\
=& \displaystyle\int_{M}  \sigma_{1}(D^2 v-b(x)L) |\nabla \p| d\mu_M\\
\leq& \displaystyle\int_{M}  \sigma_{1}(L) |\nabla \p| d\mu_M+ \displaystyle\int_{M}  \sigma_{1}(D^2 v) |\nabla \p| d\mu_M.\\
\end{split}\end{equation}
By integration by parts, the last line is equal to
\begin{equation}\begin{split}\label{2.5}
&  \displaystyle\int_{M} \sigma_{1}(L) |\nabla \p| d\mu_M-  \displaystyle\int_{M} v_{i} (|\nabla \p|)_i  d\mu_M \\
\leq &\displaystyle\int_{M} \sigma_{1}(L) |\nabla \p| d\mu_M+ \displaystyle\int_{M}  |\nabla v|\cdot |\nabla^2 \p|  d\mu_M \\
\leq &\displaystyle\int_{M}    \sigma_{1}(L) |\nabla \p|+ |\nabla^2 \p| d\mu_M.\\
\end{split}\end{equation}
Here we have used $|\nabla |\nabla \p||\leq |\nabla^2 \p|$.
Plugging (\ref{2.3})-(\ref{2.5}) into (\ref{2.2}), we get
\begin{equation}\begin{split}\label{2.8}
&\displaystyle\int_{M} v_{j} [T_1]_{ij}(\bar{D}^2\bar{V} +b(x)L) \p_i d\mu_M\\
\leq & \displaystyle\int_{M}    2(n-1)\sigma_{1}(L) |\nabla \p| + (n-1)|\nabla^2 \p|d\mu_M.\\
\end{split}\end{equation}
Thus the second term in (\ref{2.1}) is bounded by $ \int_{M}  (n-1)|\nabla^2 \p|+  2(n-1)\sigma_{1}(L) |\nabla \p| d\mu_M$.
Therefore, we conclude from (\ref{2.7}) and (\ref{2.8}) that
\begin{equation}\begin{split}
I:=&\displaystyle\int_{M} \Sigma_2(D^2 v, D^2 v) \p d\mu_M \\
=& \displaystyle\int_{M}  -v_{j} ([T_{1}]_{ij}(D^2 v))_i \p- v_{j} [T_{1}]_{ij}(D^2 v)
\p_i d\mu_M \\
\leq &\displaystyle\int_{M}\sigma_2 (L) |\p|+ 2(n-1)\sigma_{1}(L) |\nabla \p|+ (n-1)|\nabla^2 \p| d\mu_M\\
\leq &C\displaystyle\int_{M}\sigma_2 (L) |\p|+ \sigma_{1}(L) |\nabla \p|+ |\nabla^2 \p| d\mu_M,\\
\end{split}
\end{equation}
where $C$ depends only on $k$, which is equal to 2 in this section, and $n$.
This finishes the estimate of $I$.\\

To bound the term $II$ in (\ref{2.9}),
\begin{equation}\begin{split}
II:=&\displaystyle \int_M \Sigma_2(D^2 v, L) \p d\mu_M\\
= & \displaystyle \int_M v_{ij} [T_{1}]_{ij}(L) \p d\mu_M\\
= & \displaystyle \int_{M} -v_j ([T_{1}]_{ij}(L))_i \p - v_j  [T_{1}]_{ij}(L) \p_i d\mu_M.\\
\end{split}
\end{equation}
Recall that $([T_{1}]_{ij}(L))_i=0$ by the Codazzi equation. This, with by $|\nabla v|\leq 1$, implies that
\begin{equation}\begin{split}
II= & \displaystyle \int_{M} - v_j  [T_{1}]_{ij}(L) \p_i d\mu_M\\
\leq & \displaystyle \int_{M}\sigma_{1}(L) |\nabla v|\cdot |\nabla \p|d\mu_{M}\\
\leq &\displaystyle \int_{M}\sigma_{1}(L) |\nabla \p|d\mu_{M}.\\
\end{split}
\end{equation}

Finally, the estimate of term $III$ in (\ref{2.9}) is straightforward, since $\int_M \sigma_2(L)\p d\mu_M\leq \int_M \sigma_2(L)|\p|d\mu_M$.\\
In conclusion,
\begin{equation}\begin{split}
 \displaystyle \int_{M} \sigma_2(D^2v+ a L) \p d\mu_M= &\frac{1}{2}I + a\cdot II+ \frac{a^2}{2}\cdot III\\
 \leq & \displaystyle C\int_{M}(\sigma_2(L)|\p|+ \sigma_{1}(L) |\nabla \p|+ |\nabla^2 \p|) d\mu_{M}.\\
 \end{split}
\end{equation}
This completes the proof of Proposition \ref{a} when $k=2$.
\end{proof}

%\section{k=3 case of Proposition \ref{a}}
%In this section, we aim to prove
%\begin{equation}
% \displaystyle \int_{M} \sigma_3(D^2v+ a L) \p d\mu_M\leq C \int_{M} (\sigma_3(L) |\p|+ \sigma_2(L)|\nabla \p|+ \sigma_1(L)|\nabla^2 \p| + |\nabla^3 \p|) d\mu_M.\\
%\end{equation}
%First of all, we can write
%\begin{equation}\begin{split}
% \displaystyle \int_{M} \sigma_3(D^2v+ a L) \p d\mu_M= &\displaystyle  \int_{M} \left(\frac{1}{3}\Sigma_3(D^2 v, D^2 v, D^2 v)+a\Sigma_3(D^2 v, D^2v, L)+ a^2 \Sigma_3 (D^2 v, L, L)\right.\\
%&\left. + \frac{a^3}{3} \Sigma_3(L,L,L) \right) \p d\mu_M\\
% :=&\frac{1}{3}I + a II+ a^2 III+\frac{a^3}{3} IV.\\
% \end{split}
%\end{equation}

\section{General $k$ case of Proposition \ref{a}}
By the multi-linearity of $\Sigma_{k}(\cdot, ...\cdot)$, it is sufficient to
prove
\begin{equation}\label{a2}\int_{M} \Sigma_k(\overbrace{D^2v,...,D^2v}^{i_0}, L,...,L) \p d\mu_{M}\leq
C \displaystyle \int_{M} (\sigma_{k}(L)|\p|+\sigma_{k-1}(L)|\nabla \p|+...+|\nabla^k \p| )d\mu_M
\end{equation} for each $0 \leq i_0\leq k$.
In the following, we first prove (\ref{a2}) for two initial values $i_0=1 $ and $i_0=2$. We need two initial cases to start the induction argument since the index $i_0$ decreases by 2 in each induction step.

For $i_0=1$,
\begin{equation}\begin{split}
&\displaystyle \int_{M} \Sigma_k(D^2v, L,...,L) \p d\mu_{M}\\
=&\displaystyle \int_{M} v_{ij} [T_{k-1}]_{ij}(L) \p d\mu_{M}\\
=&\displaystyle  \int_{M} - v_{j} ([T_{k-1}]_{ij}(L))_i \p -v_{j} [T_{k-1}]_{ij}(L) \p_i d\mu_{M}.\\
\end{split}\end{equation}
By Lemma \ref{divT}, $ ([T_{k-1}]_{ij}(L))_i=0 $; thus
\begin{equation}\begin{split}
&\displaystyle \int_{M} \Sigma_k(D^2v, L,...,L) \p d\mu_{M}\\
=&\displaystyle  \int_{M} -v_{j} [T_{k-1}]_{ij}(L) \p_i d\mu_{M}.\\
\end{split}\end{equation}
Now $L\in \Gamma_{k+1}^+\subseteq \Gamma_{k}^+$ implies $[T_{k-1}]_{ij}(L) \geq 0$. Thus
\begin{equation}
\begin{split}
\displaystyle\int_{M} -v_{j} [T_{k-1}]_{ij}(L) \p_i d\mu_{M}\leq &\displaystyle\int_{M} Tr([T_{k-1}]_{ij}(L))|\nabla \p|\cdot |\nabla v| d\mu_{M},\\
\end{split}
\end{equation}
where $Tr([T_{k-1}]_{ij}(L))$ denotes the trace of $[T_{k-1}]_{ij}(L)$, which is, by (\ref{Trace1}), equal to $(n-k+1)\sigma_{k-1}(L)$. Hence
\begin{equation}
\begin{split}
\displaystyle\int_{M} -v_{j} [T_{k-1}]_{ij}(L) \p_i d\mu_{M}
\leq &\displaystyle C \int_{M} \sigma_{k-1}(L)|\nabla \p|d\mu_{M},\\
\end{split}
\end{equation}
where $C$ depends only on $n$ and $k$.\\

To prove the inequality (\ref{a2}) with $i_0=2$,
\begin{equation}
\begin{split}
&\displaystyle \int_{M} \Sigma_k(D^2v,D^2v, L,...,L) \p d\mu_{M}\\
=&\displaystyle \int_{M} v_{ij} [T_{k-1}]_{ij}(D^2v, L,...,L) \p d\mu_{M}\\
=&\displaystyle  \int_{M} - v_{j} ([T_{k-1}]_{ij}(D^2v, L,...,L) )_i \p -v_{j} [T_{k-1}]_{ij}(D^2v, L,...,L)  \p_i d\mu_{M}.\\
\end{split}
\end{equation}
By Lemma \ref{divTD2v}, $ ([T_{k-1}]_{ij}(D^2v, L,...,L) )_i= [T_{k-1}]_{ij}(L)L_{mi}v_m  $,
\begin{equation}\begin{split}\label{5.3}
&\displaystyle \int_{M} \Sigma_k(D^2v, D^2v, L,...,L) \p d\mu_{M}\\
=&\displaystyle  \int_{M}  -  [T_{k-1}]_{ij}(L,...,L)L_{mi}v_j v_m \p
 - [T_{k-1}]_{ij}(D^2v, L,...,L)  v_{j}\p_i d\mu_{M}.\\
\end{split}\end{equation}
For the first term on the last line of (\ref{5.3}), by (\ref{T})
$$ [T_{k-1}]_{ij}(L) L_{mi}= \sigma_{k}(L)\delta_{mj}-[T_{k}]_{mj}(L);$$ thus
we have \begin{equation}\begin{split}
& \displaystyle \int_{M} [T_{k-1}]_{ij}(L,...,L) L_{mi} v_j v_m \p d\mu_M\\
=&\displaystyle \int_{M} \sigma_{k}(L) |\nabla v|^2 \p d\mu_M
-\displaystyle \int_{M} [T_{k}]_{mj}(L) v_j v_m \p d\mu_M.\\
\end{split}
\end{equation}
Note that $|\nabla v|\leq 1$, so
$$\displaystyle \int_{M} \sigma_{k}(L) |\nabla v|^2 \p d\mu_M
\leq \displaystyle \int_{M} \sigma_{k}(L) |\p| d\mu_M.$$
Also, due to the fact that $L\in \Gamma_{k+1}^{+}$, $[T_{k}]_{mj}(L)\geq 0$. Thus
\begin{equation}\begin{split}
\displaystyle -\int_{M}[T_{k}]_{mj}(L) v_j v_m \p d\mu_M\leq C\int_{M}\sigma_{k}(L) |\nabla v|^2 |\p| d\mu_M
\leq C\int_{M}\sigma_{k}(L)  |\p| d\mu_M.\\
\end{split}
\end{equation}
For the second term in (\ref{5.3}), we first use $D^2v=\bar{D}^2v+ b(x)L$;
with $|b(x)|\leq 1$, $\bar{D}^2_{ij}v\geq 0$, $[T_{k-1}]_{ij}(L)\geq 0$ and $|\nabla v|\leq 1$, it is easy to see
\begin{equation}\begin{split}\label{5.4}
 &\displaystyle\int_{M}- [T_{k-1}]_{ij}(D^2v, L,...,L)  v_{j}\p_i d\mu_{M}\\
=&\displaystyle\int_{M}-b(x)[T_{k-1}]_{ij}(L,...,L) v_{j}\p_i -[T_{k-1}]_{ij}(\bar{D}^2v, L,...,L) v_{j}\p_i d\mu_{M}\\
\leq &\displaystyle \int_{M}  Tr([T_{k-1}]_{ij}(L))  |\nabla \p|+Tr([T_{k-1}]_{ij}(\bar{D}^2v, L,...,L)) |\nabla\p| \cdot
|\nabla v|d\mu_{M}\\
\leq &\displaystyle C \int_{M} \sigma_{k-1}(L)  |\nabla \p| +\Sigma_{k-1}(\bar{D}^2v, L,...,L)|\nabla\p|d\mu_{M}.\\
\end{split}\end{equation}
We now apply $D^2v=\bar{D}^2v+ b(x)L$ again. Then
\begin{equation}\begin{split}
&\displaystyle C \int_{M}\sigma_{k-1}(L)  |\nabla \p|+ \Sigma_{k-1}(\bar{D}^2v, L,...,L)|\nabla\p| d\mu_{M}\\
= &\displaystyle C\int_{M} (1-b(x)) \sigma_{k-1}(L) |\nabla \p| + \Sigma_{k-1}(D^2v, L,...,L)|\nabla\p|d\mu_{M}\\
\leq &\displaystyle C\int_{M}  \sigma_{k-1}(L) |\nabla \p| + \Sigma_{k-1}(D^2v, L,...,L)|\nabla\p|d\mu_{M}.\\
\end{split}\end{equation}
Now by our earlier result for $i_0=1$ in this section, $$\int_{M}\Sigma_{k}(D^2v, L,...,L)|\p|  d\mu_{M}\leq C\int_{M}\sigma_{k-1}(L)|\nabla \p|  d\mu_{M}$$ for arbitrary positive integer $k\leq n$ and any function $\p$.
In particular, this inequality holds for $k-1\leq n$ and function?? $|\nabla \p|$, namely
\begin{equation}\label{5.5}\int_{M} \Sigma_{k-1}(D^2v, L,...,L)|\nabla\p| d\mu_{M}\leq C \int_{M} \sigma_{k-2}(L) |\nabla^2\p| d\mu_{M}.
\end{equation}
Here we have used the fact $|\nabla |\nabla \p||\leq |\nabla^2 \p|$.
To conclude, by (\ref{5.4})-(\ref{5.5}), we get
\begin{equation}\begin{split}
&\displaystyle\int_{M}- [T_{k-1}]_{ij}(D^2v, L,...,L)  v_{j}\p_i d\mu_{M}\\
\leq &\displaystyle C \int_{M}  \sigma_{k-1}(L)  |\nabla \p| +\Sigma_{k-1}(\bar{D}^2v, L,...,L)|\nabla\p| d\mu_{M}\\
\leq &\displaystyle C\int_{M}   \sigma_{k-1}(L) |\nabla \p|+ \sigma_{k-2}(L)|\nabla^2\p| d\mu_{M}.\\
\end{split}\end{equation}
This finishes the estimate of the second term in (\ref{5.3}).
Therefore
\begin{equation}\begin{split}
\int_{M} \Sigma_k(D^2v, D^2v, L,...,L) \p d\mu_{M}\leq C\displaystyle \int_{M}\sigma_{k}(L)  |\p|+\sigma_{k-1}(L) |\nabla \p| +\sigma_{k-2}(L)|\nabla^2\p| d\mu_{M}.\\
\end{split}
\end{equation}
This finishes the proof of (\ref{a2}) for $i_0=2$.\\

Now we aim to prove (\ref{a2}) for $i_0=3,...,k;$ i.e.
\begin{equation}\label{5.8}
\begin{split}
I_{k,m}(\p):=& \displaystyle\int_{M} \Sigma_k(\overbrace{D^2v,..., D^2v}^m, L,...,L) \p d\mu_{M}\\
\leq &\displaystyle  C \int_M \sigma_k(L) |\p|+ \sigma_{k-1}(L)|\nabla \p|+\cdots + |\nabla^k \p|    d\mu_M,
\end{split}
\end{equation}
for some $C$ depending only on $n$ and $k$.
To begin the inductive argument, we assume (\ref{5.8}) holds for $m=1,..., i_0-1$ where $i_0\geq 3$,
which we call the inductive assumption in the following;
with this we will show (\ref{5.8}) for $m=i_0$. To simplify $I_{k,i_0}(\p)$,
we apply a similar integration by parts argument as the one to show formula (128)
%\ref{5.41}
in \cite{CW}. Such
an argument splits the estimate of $I_{k,i_0}$ into four parts.
\begin{equation}\label{5.2}
\begin{split}I_{k,i_0}(\p)= &\displaystyle (i_0-1)\frac{C_{k}^{{i_0}-2}}{kC_{k-1}^{{i_0}-2}}
\cdot I_{k,i_0-2}^{(|\nabla v|^2)}(\p)+(i_0-1)\frac{C_{k}^{{i_0}-2}}{C_{k-1}^{{i_0}-2}}
\cdot J_{k,i_0-2}^{(-1)}(\p)+(i_0-1)\frac{C_{k-1}^{{i_0}-3}}{C_{k-1}^{{i_0}-2}}\cdot K_{k,i_0-3}^{(-1)}(\p)\\
&+ N_{k,i_0-1}^{(-1)}(\p),\\
\end{split}
\end{equation}
where \begin{equation}\label{i2}I_{k,l}^{(u)}(\p):=\displaystyle \int_{M} \Sigma_{k}(\overbrace{D^2v,...,D^2v}^{l},L,...,L)u(x)\p(x) d\mu_M,\end{equation}
\begin{equation}\label{j2}J_{k,l}^{(u)}(\p):=\displaystyle\int_{M}[T_{k}]_{mj}(\overbrace{D^2v,...,D^2v}^{l},L,...,L)v_jv_m u(x)\p(x) d\mu_M,\end{equation}
\begin{equation}\label{k2}K_{k,l}^{(u)}(\p):=\displaystyle \int_{M} [T_{k-1}]_{ij}(\overbrace{D^2v,...,D^2v}^{l},L,...,L)v_{mi}v_jv_m u(x)\p(x) d\mu_M, \end{equation}
and
\begin{equation}\label{n2}N_{k,l}^{(u)}(\p):=\displaystyle \int_{M} [T_{k-1}]_{ij}(\overbrace{D^2v,...,D^2v}^{l},L,...,L)v_j\p_i u(x) d\mu_M. \end{equation}
We remark that in the above definitions, $\p(x)$ is the test
function that has appeared in the statement of the main theorem, while $u(x)$ is a bounded coefficient function which may vary from line to line in our later argument.

In the following we will call any term that takes the form $I_{k,l}^{(u)}(\p)$, $J_{k,l}^{(u)}(\p)$, $K_{k,l}^{(u)}(\p)$, $N_{k,l}^{(u)}(\p)$ the $I$-type term, the $J$-type term, the $K$-type and the $N$-type term respectively. In the special case when $u=1$, we will denote $I_{k,l}^{(1)}(\p)$, $J_{k,l}^{(1)}(\p)$, $K_{k,l}^{(1)}(\p)$, $N_{k,l}^{(1)}(\p)$ by $I_{k,l}(\p)$, $J_{k,l}(\p)$, $K_{k,l}(\p)$, $N_{k,l}(\p)$ for simplicity.\\
In order to prove (\ref{5.8}) we need to estimate the $I$-type term, the $J$-type term, the $K$-type and the $N$-type term individually. The main idea of the proof is that each of the four terms in (\ref{5.2}) is of an decreased index ($i_0-1$, $i_0-2$ or $i_0-3$); if we can bound them by the $I$-type terms with indices strictly less than $i_0$, then we can apply the inductive assumption. We will show both the $I$-type term and the $J$-type term are bounded by $\sum_{s\leq l}I_{k,s}(\p)$; the $N$-type term is bounded by $\sum_{s\leq l}I_{k,s}(\nabla\p)$; and the $K$-type term is inductively bounded by the $K$-type term $\sum_{s\leq l}K_{k,s}(\p)$ and the $N$-type term $\sum_{s\leq l}N_{k,s}(\p)$, thus bounded by
$$\sum_{s\leq l}I_{k,s}^{(1)}(\p)+\sum_{s\leq l}I_{k,s}^{(1)}(\nabla\p).$$

We begin by looking at the $I$-type term, the $J$-type term. They can be estimated using a similar argument as the ones proved in Lemma 6.3 and Claim 2 in \cite{CW}. We present the results here without proof.

\textbf{\textit{Proposition I}}: \textit{For any bounded function $u(x)$, let us denote $\max_{x\in M} |u(x)|$ by $U$.
Then for any $l\geq 0$ and any function $\p$, there exist positive constants
$A_0,..., A_{l}$ depending on $U$, $k$, and $n$, such that
\begin{equation}\label{I2}
\begin{split}
I_{k,l}^{(u)}(\p)\leq \displaystyle \sum_{s=0}^{l}A_s I_{k,s}(|\p|).\\
\end{split}
\end{equation}
In particular, one can choose $A_l=U$.}

\textbf{\textit{Proposition J}}: \textit{For any bounded function $u(x)$, let us denote $\max_{x\in M} |u(x)|$ by $U$.
Then for any $l\geq 0$ and any function $\p$, there exist positive constants
$A_0,..., A_{l}$ depending on $U$, $k$, and $n$, such that
\begin{equation}\label{J2}
\begin{split}
J_{k,l}^{(u)}(\p)\leq \displaystyle \sum_{s=0}^{l} A_s I_{k,s}(|\p|).\\
\end{split}
\end{equation}
}

On the other hand, the $K$-type and $N$-type estimates are quite different from those in \cite{CW}.
They will be the focus of the argument below. We begin by proving the $N$-type estimate first.

\textbf{\textit{Proposition N}}: \textit{For any bounded function $u(x)$, let us denote $\max_{x\in M} |u(x)|$ by $U$.
Then for any $l\geq 0$ and any function $\p$, there exist positive constants
$\tilde{A}_0,..., \tilde{A}_{l}$ depending on $U$, $k$, and $n$, such that
\begin{equation}\label{N2}
\begin{split}
N_{k,l}^{(u)}(\p)\leq \displaystyle \sum_{s=0}^{l} \tilde{A}_s I_{k-1,s}(|\nabla \p|).\\
\end{split}
\end{equation}
}

We present the whole proof of the $N$-type estimate in the following, since this type of estimate has not appeared
in \cite{CW}.

\begin{proof}
Recall that
\begin{equation}
N_{k,l}^{(u)}(\p):= \displaystyle \int_{M} [T_{k-1}]_{ij}(\overbrace{D^2 v,...,D^2 v}^{l},L,...,L) v_j\p_i u(x)d\mu_M. \end{equation}
By $D^2 v= \bar{D}^2 V + b(x)L$ with $|b(x)|\leq 1$, we have
\begin{equation}
\begin{split}
N_{k,l}^{(u)}(\p)=
&\displaystyle \sum_{s=0}^{l} \int_M b_s(x)[T_{k-1}]_{ij}(\overbrace{\bar{D}^2 v,...,\bar{D}^2 v}^{s},L,...,L) v_j\p_i u(x)d\mu_M,\\
\end{split}
\end{equation}
where $b_s(x)$ are some bounded functions with bounds only depending on $n$ and $k$.
Notice
$$[T_{k-1}]_{ij}(\overbrace{\bar{D}^2 v,...,\bar{D}^2 v}^{s},L,...,L)\geq 0, $$
and $|\nabla v|\leq 1$. Thus
\begin{equation}
\begin{split}
&[T_{k-1}]_{ij}(\overbrace{\bar{D}^2 v,...,\bar{D}^2 v}^{s},L,...,L)v_j \p_i u\\
\leq& U\cdot Tr([T_{k-1}]_{ij}(\overbrace{\bar{D}^2 v,...,\bar{D}^2 v}^{s},L,...,L) )\cdot |\nabla \p|\\
=&U\cdot \frac{n-(k-1) }{k-1}  \Sigma_{k-1}(\overbrace{\bar{D}^2 v,...,\bar{D}^2 v}^{s},L,...,L) \cdot |\nabla\p|.
\end{split}
\end{equation}
\begin{equation}
\begin{split}\label{5.16}
N_{k,l}^{(u)}(\p)\leq &
\displaystyle \sum_{s=0}^{l} \tilde{A}_s\cdot\int_M \Sigma_{k-1}(\overbrace{\bar{D}^2 v,...,\bar{D}^2 v}^{s},L,...,L) \cdot|\nabla\p| d\mu_M.\\
\end{split}
\end{equation}
Here $\tilde{A}_s$ are constants only depending on $U$, $k$, and $n$.
We then apply $D^2 v= \bar{D}^2 V + b(x)L$ again. By the multi-linearity of $\Sigma_{k-1}(\cdot, ..., \cdot)$,
\begin{equation}
\begin{split}
&\displaystyle \sum_{s=0}^{l} \tilde{A}_s\cdot\int_M \Sigma_{k-1}(\overbrace{\bar{D}^2 v,...,\bar{D}^2 v}^{s},L,...,L) \cdot|\nabla\p| d\mu_M\\
=& \displaystyle \sum_{s=0}^{l} \tilde{b}_s(x)\cdot\int_M \Sigma_{k-1}(\overbrace{D^2 v,...,D^2 v}^{s},L,...,L) \cdot|\nabla\p| d\mu_M\\
= & \displaystyle \sum_{s=0}^{l} I_{k-1,s}^{(\tilde{b}_s(x))}(|\nabla \p|),
\end{split}
\end{equation}
where $\tilde{b}_s(x)$ are bounded functions with bounds only depending on $U$, $k$, and $n$.\\
By Proposition I,
\begin{equation}
\displaystyle \sum_{s=0}^{l} I_{k-1,s}^{(\tilde{b}_s(x))}(|\nabla \p|)\leq
\displaystyle \sum_{s=0}^{l} \tilde{A}_s I_{k-1,s}(|\nabla \p|).
\end{equation}
Here $\tilde{A}_s$ are positive constants which are different from the ones in (\ref{5.16}). But again they only depend on the bounds of $\tilde{b}_s(x)$, $n$ and $k$; thus they only depend on $U$, $k$, and $n$.
In conclusion,
\begin{equation}
\begin{split}
N_{k,l}^{(u)}(\p)\leq
  \displaystyle \sum_{s=0}^{l} \tilde{A}_s I_{k-1,s}(|\nabla \p|),
\end{split}
\end{equation}
for some $\tilde{A}_s$ only depending on $U$, $k$, and $n$.
This ends the proof of Proposition N.
\end{proof}

\textbf{\textit{Proposition K}}: \textit{For any bounded function $u(x)$, let us denote $\max_{x\in M} |u(x)|$ by $U$.
Then for any $i_0\geq 3$ and any function $\p$, there exist positive constants
$A_0,..., A_{i_0-3}$, and $\tilde{A}_0,..., \tilde{A}_{i_0-3}$ depending on $U$, $k$, and $n$, such that
\begin{equation}\label{K2}
\begin{split}
K_{k,i_0-3}^{(-1)}(\p)\leq \displaystyle \sum_{s=0}^{i_0-2} A_s I_{k,s}(|\p|)+\sum_{s=0}^{i_0-3} \tilde{A}_s I_{k-1,s}(|\nabla \p|).\\
\end{split}
\end{equation}
}
Before proving Proposition K, we first show the following two inequalities.
\begin{lemma}\label{k0}Let $v$ be a function on $M$ with $|\nabla v|\leq 1$.
For any integer $3\leq i_0\leq k $,
\begin{equation}\begin{split}\label{5.11}
K_{k,0}^{(\pm | \nabla v|^{i_0-3})}(\p) \leq &\sum_{s=0}^{1}A_s I_{k,s}(|\p|)+ \tilde{A}_0 I_{k-1,0}(|\nabla \p|)\\
=&A_0 I_{k,0}(|\p|) + A_1 I_{k,1}(|\p|)+ \tilde{A}_0 I_{k-1,0}(|\nabla \p|), \quad \mbox{when $i_0$ is odd};
\end{split}
\end{equation}
and
\begin{equation}\begin{split}\label{5.12}
K_{k,1}^{(\pm|\nabla v|^{i_0-4})}(\p)\leq &\sum_{s=0}^{2}A_s I_{k,s}(|\p|)+ \sum_{s=0}^{1}\tilde{A}_s I_{k-1,s}(|\nabla \p|), \quad \mbox{when $i_0$ is even}.\\
\end{split}
\end{equation}
\end{lemma}

\begin{proof}
To prove (\ref{5.11}) when $i_0$ is odd, we first write
\begin{equation}
\begin{split}
K_{k,0}^{(\pm|\nabla v|^{i_0-3})}(\p):=&\displaystyle \pm\int_{M} [T_{k-1}]_{ij}(L,...,L)v_{mi}v_jv_m |\nabla v|^{i_0-3} \p d\mu_M\\
=&\displaystyle \pm\int_{M}[T_{k-1}]_{ij}(L,...,L)\frac{1}{i_0-1}v_j(|\nabla v|^{i_0-1})_{i} \p d\mu_M\\
=&\displaystyle \mp\int_{M}([T_{k-1}]_{ij}(L,...,L) )_i \frac{1}{i_0-1}v_j |\nabla v|^{i_0-1}\p d\mu_M \\
 &\displaystyle \mp\int_{M}[T_{k-1}]_{ij}(L,...,L) \frac{1}{i_0-1}v_{ij} |\nabla v|^{i_0-1} \p d\mu_M\\
 &\displaystyle \mp \int_{M}[T_{k-1}]_{ij}(L,...,L) \frac{1}{i_0-1}v_j \p_i |\nabla v|^{i_0-1} d\mu_M. \\
\end{split}
\end{equation}
Notice that by Lemma \ref{divT},
\begin{equation}
([T_{k-1}]_{ij}(L,...,L) )_i=0.
\end{equation}\\
So we only need to estimate the rest two terms. First of all,
\begin{equation}\begin{split}
\displaystyle
&\mp\int_{M}[T_{k-1}]_{ij}(L,...,L) \frac{1}{i_0-1}v_{ij} |\nabla v|^{i_0-1} \p d\mu_M\\
=&\mp\frac{1}{i_0-1}\int_{M}\Sigma_{k}(D^2v,L,...,L) |\nabla v|^{i_0-1} \p d\mu_M\\
= &\mp \frac{1}{i_0-1}I_{k,1}^{(|\nabla v|^{i_0-1})}(\p) ,\\
\end{split}\end{equation}
by the definition of $I_{k,l}^{(u)}(\p)$ in (\ref{i2}).
Now by the $I$-type estimate proved in Proposition I,
\begin{equation}
\begin{split}
\mp \frac{1}{i_0-1}I_{k,1}^{(|\nabla v|^{i_0-1})}(\p)\leq\sum_{s=0}^{1}A_s I_{k,s}(|\p|)
=A_0 I_{k,0}(|\p|)+ A_1 I_{k,1}(|\p|),
\end{split}
\end{equation} for some constants $A_s$ depending only on $n$ and $k$.

Another term  $\mp \displaystyle \int_{M}[T_{k-1}]_{ij}(L,...,L) \frac{1}{i_0-1}v_j \p_i |\nabla v|^{i_0-1} d\mu_M$
is an $N$-type term. In fact,
\begin{equation}\begin{split}
\displaystyle \mp \int_{M}[T_{k-1}]_{ij}(L,...,L) \frac{1}{i_0-1}v_j \p_i |\nabla v|^{i_0-1} d\mu_M
 =\displaystyle \mp\frac{1}{i_0-1} N_{k,0}^{(|\nabla v|^{i_0-1})}(\p).
\end{split}
\end{equation}
Therefore by Proposition N, this term is bounded by
\begin{equation}
\tilde{A}_0 I_{k-1,0}(|\nabla \p|)=\tilde{A}_0 \displaystyle \int_{M} \sigma_{k-1}(L)|\nabla \p|d\mu_M.
\end{equation}
The estimates of these two terms lead to
\begin{equation}
\begin{split}
K_{k,0}^{(\pm |\nabla v|^{i_0-3})}(\p)\leq &\sum_{s=0}^{1}A_s I_{k,s}(|\p|)+ \tilde{A}_0  I_{k-1,0}(|\nabla \p|)\\
=&A_0 I_{k,0}(|\p|) + A_1 I_{k,1}(|\p|)+ \tilde{A}_0 I_{k-1,0}(|\nabla \p|).\\
\end{split}
\end{equation}
This finishes the proof of (\ref{5.11}).\\

To prove (\ref{5.12}) when $i_0$ is even, we write
\begin{equation}\label{5.13}
\begin{split}
K_{k,1}^{(\pm |\nabla v|^{i_0-4})}(\p):=&\pm \displaystyle \int_{M} [T_{k-1}]_{ij}(D^2v, L,...,L)v_{mi}v_jv_m  |\nabla v|^{i_0-4}
\p d\mu_M\\
=&\pm \displaystyle \int_{M}[T_{k-1}]_{ij}(D^2v, L,...,L)\frac{1}{i_0-2}v_j (|\nabla v|^{i_0-2})_i \p d\mu_M\\
=&\mp \displaystyle \int_{M}([T_{k-1}]_{ij}(D^2v, L,...,L) )_i \frac{1}{i_0-2}v_j  |\nabla v|^{i_0-2}\p d\mu_M\\ &\mp\int_{M}[T_{k-1}]_{ij}(D^2v, L,...,L) \frac{1}{i_0-2}v_{ij}  |\nabla v|^{i_0-2}\p d\mu_M\\
&\mp\int_{M}[T_{k-1}]_{ij}(D^2v, L,...,L) \frac{1}{i_0-2}v_{j} \p_i |\nabla v|^{i_0-2} d\mu_M.\\
\end{split}
\end{equation}
For the first term in the last equality of (\ref{5.13})
$$\mp \int_{M}([T_{k-1}]_{ij}(D^2v, L,...,L) )_i \frac{1}{i_0-2}v_j |\nabla v|^{i_0-2} \p d\mu_M,$$
we recall Lemma \ref{divTD2v}
$$([T_{k-1}]_{ij}(D^2v, L,...,L) )_i= -[T_{k-1}]_{ij}(L)L_{mi} v_m. $$
Thus \begin{equation}\label{5.14}
\begin{split}&\mp\int_{M}([T_{k-1}]_{ij}(D^2v, L,...,L) )_i\frac{1}{i_0-2} v_j |\nabla v|^{i_0-2} \p d\mu_M\\
=& \pm\frac{1}{i_0-2}\int_{M}[T_{k-1}]_{ij}(L)L_{mi} v_l v_m |\nabla v|^{i_0-2}\p d\mu_M.\\
\end{split}
\end{equation}
By formula (\ref{T}), and the definition of $I_{k,l}^{(u)}(\p)$, $J_{k,l}^{(u)}(\p)$ in (\ref{i2}), (\ref{j2})
\begin{equation}\label{5.15}
\begin{split}
&\pm\frac{1}{i_0-2}\int_{M}[T_{k-1}]_{ij}(L,...,L)L_{mi} v_m v_j |\nabla v|^{i_0-2}\p d\mu_M\\=&\displaystyle \int_{M}\{\pm C_1\Sigma_{k}(L,...,L)\delta_{jl}\mp C_2[T_{k}]_{jl}(L,...,L)\} v_l v_j |\nabla v|^{i_0-2}\p d\mu_M\\
=& \int_{M}\pm C_1\Sigma_{k}(L,...,L)|\nabla v|^{i_0}\p \mp C_2[T_{k}]_{jl}(L,...,L) v_l v_j |\nabla v|^{i_0-2}\p d\mu_M\\
=& \pm C_1 I_{k,0}^{(|\nabla v|^{i_0})}(\p)\mp C_2 J_{k,0}^{(|\nabla v|^{i_0-2})}(\p),\\
\end{split}
\end{equation}
where $C_1$, $C_2$ are positive constants depending only on $n$ and $k$.
Notice $|\nabla v|\leq 1$; thus by Proposition I and Proposition J, the $I$-type term $\pm C_1I_{k,0}^{(|\nabla v|^{i_0})}(\p)$ and the $J$-type term
$\mp  C_2 J_{k,0}^{(|\nabla v|^{i_0-2})}(\p)$ are both bounded by $\displaystyle A_0 I_{k,0}(|\p|) $ for some positive constants $A_0$, namely
\begin{equation}\label{5.17}\pm C_1I _{k,0}^{(|\nabla v|^{i_0})}(\p) \mp C_2 J_{k,0}^{(|\nabla v|^{i_0-2})}(\p) \leq A_0 I_{k,0}(|\p |).\\\end{equation}
By (\ref{5.14})-(\ref{5.17}), we get $$\mp \int_{M}([T_{k-1}]_{ij}(D^2v, L,...,L) )_i\frac{1}{i_0-2} v_j |\nabla v|^{i_0-2}\p d\mu_M\leq A_0 I_{k,0}(|\p |).$$ This completes the estimate of the term $\mp \int_{M}([T_{k-1}]_{ij}(D^2v, L,...,L) )_i \frac{1}{i_0-2}v_j |\nabla v|^{i_0-2}\p d\mu_M$ in (\ref{5.13}).\\
Next we need to estimate the second term in the last equality of (\ref{5.13}).
Notice
\begin{equation}
\begin{split}\displaystyle
&\mp\int_{M}[T_{k-1}]_{ij}(D^2v, L,...,L) \frac{1}{i_0-2}v_{ij} |\nabla v|^{i_0-2} \p d\mu_M\\
=&\mp \frac{1}{i_0-2}\int_{M}\Sigma_{k}(D^2v,D^2 v, L,...,L) |\nabla v|^{i_0-2}\p d\mu_M\\
=&\mp \frac{1}{i_0-2}I_{k,2}^{(|\nabla v|^{i_0-2})}(\p),\\
\end{split}
\end{equation}
by the definition of $I_{k,l}^{(u)}(\p)$ in (\ref{i2}).
Thus by the $I$-type estimate proved in Proposition I, $$ \mp\frac{1}{i_0-2}I_{k,2}^{(|\nabla v|^{i_0-2})}(\p) \leq \displaystyle \sum_{s=0}^{2}A_s I_{k,s}(|\p|).$$
Finally we estimate the last term $\mp\int_{M}[T_{k-1}]_{ij}(D^2v, L,...,L) \frac{1}{i_0-2}v_{j} \p_i |\nabla v|^{i_0-2} d\mu_M
$ in (\ref{5.13}).
\begin{equation}
\begin{split}
\mp\int_{M}[T_{k-1}]_{ij}(D^2v, L,...,L) \frac{1}{i_0-2}v_{j} \p_i |\nabla v|^{i_0-2} d\mu_M=\mp \frac{1}{i_0-2} N_{k, 1}^{(|\nabla v|^{i_0-2})}(\p).\\
\end{split}
\end{equation}
Thus by Proposition N, this is bounded by $\sum_{s=0}^{1}\tilde{A}_s I_{k-1,s}(|\nabla \p|)$.
By the estimates of the three terms in (\ref{5.13}), we conclude that
$$K_{k,1}^{(\pm|\nabla v|^{i_0-4})}(\p)\leq \sum_{s=0}^{2}A_s I_{k,s}(|\p|)+ \sum_{s=0}^{1}\tilde{A}_s I_{k-1,s}(|\nabla \p|).$$
\end{proof}

\begin{proof} of Proposition K:
If $i_0=3$ or $4$, $K_{k,i_0-3}^{(-1)}(\p)$ is equal to either $K_{k,0}^{(-1)}(\p)$ or $K_{k,1}^{(-1)}(\p)$.
The estimates of these two terms have already been proved in inequality (\ref{5.11}) with $i_0=3$
and (\ref{5.12}) with $i_0=4$ respectively; thus we assume $i_0\geq 5$ from now on.
To estimate the $K$-type terms $K_{k,i_0-3}^{(-1)}(\p)$ for $i_0\geq 5$, we first apply a similar argument as the one to derive formula (154) in \cite{CW}. This implies
\begin{equation}\begin{split}\label{5.9}
K_{k,i_0-3}^{(-1)}(\p) =& \frac{1}{2}I_{k,i_0-2}^{(|\nabla v|^2)}(\p)-C_1 I_{k,i_0-4}^{(|\nabla v|^4)}(\p)
+C_2 J_{k,i_0-4}^{(|\nabla v|^2)}(\p)\\
&+ C_3 K_{k,i_0-5}^{(|\nabla v|^2)}(\p)+ \frac{1}{2} N_{k,i_0-3}^{(|\nabla v|^2)}(\p).\\
\end{split}
\end{equation}
Here $C_1$, $C_2$,  $C_3$ are positive constants depending only on $n$ and $k$.
For detailed steps, one can refer to the similar argument (156)-(161) present in \cite{CW}.
By Proposition I, Proposition J, and Proposition N, there exist positive constants $A_s$ and $\tilde{A}_s$ for
$s=0,...,i_0-3$ depending only on $k$, $n$, $C_1$, $C_2$ and $\max_{x\in M} |\nabla v(x)|\leq 1$, thus depending
only on $n$ and $k$, such that
\begin{equation}
\frac{1}{2}I_{k,i_0-2}^{(|\nabla v|^2)}(\p)\leq \sum_{s=0}^{i_0-2} A_s I_{k,s}(|\p|).
\end{equation}
\begin{equation}
-C_1 I_{k,i_0-4}^{(|\nabla v|^4)}(\p)\leq \sum_{s=0}^{i_0-4} A_s I_{k,s}(|\p|).
\end{equation}

\begin{equation}
C_2 J_{k,i_0-4}^{(|\nabla v|^2)}(\p)\leq \sum_{s=0}^{i_0-4} A_s I_{k,s}(|\p|).
\end{equation}

\begin{equation}
\displaystyle \frac{1}{2} N_{k,i_0-3}^{(|\nabla v|^2)}(\p)\leq \sum_{s=0}^{i_0-3} \tilde{A}_s I_{k-1,s}(|\nabla \p|).
\end{equation}

Here $A_s$ in each inequality may be different.
By these inequalities, (\ref{5.9}) is deduced to
\begin{equation}\begin{split}\label{5.6}
K_{k,i_0-3}^{(-1)}(\p) \leq &\sum_{s=0}^{i_0-2} A_s I_{k,s}(|\p|)
 + \sum_{s=0}^{i_0-3} \tilde{A}_s I_{k-1,s}(|\nabla \p|) + C_3 K_{k,i_0-5}^{(|\nabla v|^2)}(\p).\\
\end{split}
\end{equation}

The argument stops if either $i_0-5= 0$ or $i_0-5=1$; otherwise we perform similar arguments to
$K_{k,i_0-5}^{(|\nabla v|^2)}(\p)$ to get
\begin{equation}\begin{split}\label{5.7}
K_{k,i_0-5}^{(|\nabla v|^2)}(\p) \leq \sum_{s=0}^{i_0-4} A_s I_{k,s}(|\p|)
 + \sum_{s=0}^{i_0-5} \tilde{A}_s I_{k-1,s}(|\nabla \p|) + C_3 K_{k,i_0-7}^{(-|\nabla v|^4)}(\p).
\end{split}
\end{equation}
We remark here that the constant $C_3$ in (\ref{5.7}) may be different from the one in (\ref{5.6}). But
they are both positive constants depending only on $n$ and $k$, so we use the same notation when
it is not necessary to distinguish them.

Such an inductive argument will stop at the $q$-th step, where $q=\frac{[i_0-3]}{2}$. If $i_0$
is odd, then when the induction stops we get
\begin{equation}
\begin{split}
K_{k,i_0-3}^{(-1)}(\p) \leq &\sum_{s=0}^{i_0-2} A_s I_{k,s}(|\p|)
 + \sum_{s=0}^{i_0-3} \tilde{A}_s I_{k-1,s}(|\nabla \p|) \\
 &+ C_3 K_{k,0}^{((-1)^{\frac{i_0-1}{2}}\cdot |\nabla v|^{i_0-3})}(\p).\\
 \end{split}
\end{equation}
\noindent If $i_0$ is even, then when the induction stops we get
\begin{equation}
\begin{split}
K_{k,i_0-3}^{(-1)}(\p) \leq &\sum_{s=0}^{i_0-2} A_s I_{k,s}(|\p|)
 + \sum_{s=0}^{i_0-3} \tilde{A}_s I_{k-1,s}(|\nabla \p|)\\
  &+ C_3 K_{k,1}^{((-1)^{\frac{i_0-2}{2}}\cdot |\nabla v|^{i_0-4})}(\p).\\
  \end{split}
\end{equation}
By inequalities (\ref{5.11}) and (\ref{5.12}) with the inductive formula (\ref{5.7}),
we conclude that
$$K_{k,i_0-3}^{(-1)}\leq \sum_{s=0}^{i_0-2}A_s I_{k,s}(|\p|)+\sum_{s=0}^{i_0-3}\tilde{A}_s I_{k-1,s}(|\nabla \p|).$$
This finishes the proof of Proposition K.
\end{proof}

We are now ready to apply these four types of estimates to show (\ref{a}) for $m=i_0$.
With Proposition I, J and K, and the inductive formula (\ref{5.2}), we obtain
\begin{equation}
\begin{split}
\displaystyle I _{k,i_0}(\p):=&\int_M\Sigma_k(\overbrace{D^2v,..., D^2v}^{i_0},L,...,L)\p d\mu_M\\
\leq & \displaystyle \sum_{s=0}^{i_0-2}A_s\int_M \Sigma_k(\overbrace{D^2v,..., D^2v}^s,L,...,L)|\p|d\mu_M\\
 &+ \displaystyle \sum_{s=0}^{i_0-1}\tilde{A}_s\int_M \displaystyle \Sigma_{k-1}(\overbrace{D^2v,..., D^2v}^s,L,...,L)|\nabla \p| d\mu_M. \\
\end{split}
\end{equation}
The first sum above is equal to
$$ \sum_{s=0}^{i_0-2}A_s I_{k,s}(|\p|); $$
and the second sum is equal to
$$  \sum_{s=0}^{i_0-1} \tilde{A}_s I_{k,s}(|\nabla \p|).$$
As the index $s$ has dropped below $i_0$, both sums can be bounded by using the inductive
assumption, i.e. (\ref{5.8}) holds for $1\leq m\leq i_0-1$ and any function $\p$.
%Thus if we apply this induction relation again, the first sum in the last line above is bounded by two sums:
%\begin{equation}
%\begin{split}
% & \displaystyle\sum_{s=0}^{i_0-2} \int_M \Sigma_k(\overbrace{D^2v,..., D^2v}^s,L,...,L)|\p|d\mu_M \\
%\leq & \displaystyle \sum_{s=0}^{i_0-4} \int_M \Sigma_k(\overbrace{D^2v,..., D^2v}^s,L,...,L)|\p| d\mu_M+  \displaystyle \sum_{s=0}^{i_0-3}\int_M\Sigma_{k-1}(\overbrace{D^2v,..., D^2v}^s,L,...,L)|\nabla \p|d\mu_M; \\
%\end{split}
%\end{equation}
%and the second sum is bounded by
% \begin{equation}
%\begin{split}
% &\displaystyle \sum_{s=0}^{i_0-1}\int_M\Sigma_{k-1}(\overbrace{D^2v,..., D^2v}^s,L,...,L)|\nabla \p|d\mu_M\\
%\leq &\displaystyle \sum_{s=0}^{i_0-3} \int_M \Sigma_k(\overbrace{D^2v,..., D^2v}^s,L,...,L)|\nabla \p|d\mu_M+ \displaystyle \sum_{s=0}^{i_0-2}\int_M \Sigma_{k-1}(\overbrace{D^2v,..., D^2v}^s,L,...,L)|\nabla^2 \p|d\mu_M. \\
%\end{split}
%\end{equation}
%If we continue this induction, as the index $s$ finally drops to 0, we will get
Therefore we have
\begin{equation}
\begin{split}
&\displaystyle \int_M\Sigma_k(\overbrace{D^2v,..., D^2v}^{i_0},L,...,L)\p d\mu_M\\
\leq &\displaystyle \sum_{s=0}^{i_0-2}A_s I_{k,s}(|\p|)+\sum_{s=0}^{i_0-1} \tilde{A}_s I_{k,s}(|\nabla \p|)\\
%& C\displaystyle \int_M(\Sigma_k(L,...,L)|\p| + \Sigma_{k-1}(L,...,L)|\nabla \p|+\cdots + |\nabla^k \p|)d\mu_M\\
\leq & C \displaystyle \int_M (\sigma_k(L)|\p| + \sigma_{k-1}(L)|\nabla \p|+\cdots +
+|\nabla^k \p|)d\mu_M,
\end{split}
\end{equation}
where $C$ depends only on $n$ and $k$.
This is the conclusion that we aim to prove in this section.


\begin{thebibliography}{0}


\bibitem{Alex1} A.D. Alexandrov; Zur Theorie der gemischten Volumina von konvexen K\"{o}rpern, II. Neue Ungleichungen zwischen
den gemischten Volumina und ihre Anwendungen, Mat. Sb. (N.S.) 2 (1937), 1205-1238 (in Russian).

\bibitem{Alex2} A.D. Alexandrov; Zur Theorie der gemischten Volumina von konvexen K\"{o}rpern, III. Die Erweiterung zweeier
Lehrsatze Minkowskis \"{u}ber die konvexen Polyeder auf beliebige konvexe Flachen, Mat. Sb. (N.S.) 3 (1938), 27-46 (in Russian).

\bibitem{ADM}S. Alesker, S. Dar and V. Milman; A remarkable measure preserving
 diffeomorphism between two convex bodies in $\mathbb{R}^n$, Geom. Dedicata 74 (1999), 201-212.

\bibitem{Brenier}Y. Brenier; Polar factorization and monotone rearrangement of
 vector-valued functions, Comm. Pure Appl. Math. 44 (1991), 375-417.

\bibitem{Caff1}L.A. Caffarelli; Boundary regularity of maps with convex potentials,
Comm. Pure Appl. Math. 45 (1992); 1141-1151.

\bibitem{Caff2}L.A. Caffarelli; The regularity of mappings with a convex potential,
J. Amer. Math. Soc. 5 (1992), 99-104.

\bibitem{Caff3}L.A. Caffarelli; Boundary regularity of maps with convex potentials.
II, Ann. Math. 144 (1996), 453-496.

\bibitem{Caff4}L.A. Caffarelli, L. Nirenberg, and J. Spruck; The Dirichlet problem for nonlinear second order elliptic equations,
III: Functions of the eigenvalues of the Hessian, Acta Math. 155 (1985), 261-301.

\bibitem{Castillon}P. Castillon; Submanifolds, isoperimetric inequalities and optimal
 transportation, J. Funct. Anal. 259 (2010), 79-103.

\bibitem{Chang-Wang2010} S.Y. Chang, Y. Wang; On Aleksandrov-Fenchel Inequalities for k-Convex domains, Milan Journal
of Mathematics: Volume 79, Issue 1 (2011), Page 13-38.

\bibitem{CW} S.Y. Chang, Y. Wang; Inequalities for quermassintegrals on k-Convex domains, in submission.


\bibitem{Chavel}I. Chavel; Isoperimetric inequalities, Cambridge Tracts in Math., vol.
 145, Cambridge University Press, Cambridge, 2001.

\bibitem{Evans-Spruck}L.C. Evans, J. Spruck; Motion of level sets by mean curvature I,
 J. Differential Geom. 33 (1991), 635-681.

\bibitem{Garding}L. Garding; An inequality for hyperbolic polynomials, J. Math. Mech. 8 (1959), 957-965.

\bibitem{Gerhardt}C. Gerhardt; Flow of nonconvex hypersurfaces into spheres, J. Differential Geom. 32 (1990), 299-314.

\bibitem{Guan-Li}P. Guan, J. Li; The quermassintegral inequalities for $k$-convex
 star-shaped domains, Adv. Math. 221 (2009), 1725-1732.

\bibitem{Guan-Wang}P. Guan, G. Wang; Geometric inequalities on locally conformally flat manifolds, Duke Math. J. 124 (2004), 177-212.

\bibitem{H}L. Hormander; Notions of Convexity, Inequalities, Birkh\"{a}user Boston, Boston, 1994.

\bibitem{HL}G.H. Hardy, J.E. Littlewood, G. Polya; Inequalities, Cambridge Univ. Press, Cambridge, 1934.

\bibitem{Huisken1}G. Huisken; Flow by mean curvature of convex surfaces into spheres, J. Differential Geom. 20 (1984), 237-266.

\bibitem{Huisken2}G. Huisken; T. Ilmanen, The inverse mean curvature flow and the Riemannian Penrose inequality, J. Differential
Geom. 59 (2001), 353-437.

\bibitem{Huisken3}G. Huisken; C. Sinestrari, Convexity estimates for mean curvature flow and singularities of mean convex surfaces,
Acta Math. 183 (1999), 45-70.

\bibitem{loeper}G. Loeper; On the regularity of maps solutions of optimal transportation
problems, Acta Math. 202 (2009), no. 2, 241-283.

\bibitem{McCann2}R. McCann; Existence and uniqueness of monotone measure-preserving maps, Duke Math. J. 80 (1995), 309-323.

\bibitem{McCann}R. McCann; A convexity principle for interacting gases, Adv. Math. 128 (1997), 153-179.

\bibitem{Mk}H. Minkowski; Theorie der konvexen K\"{o}rper, insbesondere Begr\"{u}ndung ihres Oberfl\"{a}chenbegriffs. Ges, Abh., Leipzig-Berlin 1911, 2, 131-229.
\bibitem{MS}J.H. Michael, L.M. Simon; Sobolev and mean-value inequalities on generalized submanifolds of $\mathbb{R}^{n}$, Comm. Pure Appl. Math. 26 (1973), 361-379.

\bibitem{MTW}X.N. Ma, N. Trudinger, X.J. Wang; Regularity of potential functions of the
optimal transportation problem, Arch. Rational Mech. Anal. 177 (2005), 151-183.

\bibitem{Reilly}R. Reilly; On the Hessian of a function and the curvatures of its graph, Michigan Math. J. 20 (1973), 373-383.

\bibitem{Tru}N. Trudinger; Isoperimetric inequalities for quermassintegrals, Ann. Inst. H. Poincar\`{e} Anal. Non Lin\`{e}aire
11 (1994), 411-425.

\bibitem{Urbas}J. Urbas; On the expansion of starshaped hypersurfaces by symmetric functions
of their principal curvatures, Math. Z. 205 (1990), 355-372.

\bibitem{Vi1}C. Villani; Topics in optimal transportation, Graduate studies in mathematics,
vol. 58. American Mathematical Society, Providence (2003).

\bibitem{Vi2}C. Villani; Optimal transport : old and new, Grundlehren Math. Wiss. 338, Springer, Berlin, 2009.

\bibitem{Vi3}D. Cordero-Erausquin, B. Nazaret, C. Villani; A mass-transportation approach to sharp
 Sobolev and Gagliardo-Nirenberg inequalities, Adv. Math. 182 (2004), 307-322.


\end{thebibliography}
\end{document}